\documentclass[12pt,reqno]{amsart}
\usepackage{eurosym}
\usepackage{amsfonts}
\usepackage{amsmath}
\usepackage{tikz}
\usetikzlibrary{positioning,calc}
\usepackage{pgfplots}
\usepackage{float}
\usepackage{booktabs}
\usepackage{amssymb,latexsym}
\usepackage{enumerate}
\usepackage[bookmarksnumbered, colorlinks, plainpages]{hyperref}
\usepackage{amsbsy}
\usepackage{amscd}
\usepackage{graphicx}
\usepackage{epsfig}
\usepackage{tabularx}           
\usepackage[export]{adjustbox}  
\usepackage{fancyhdr}
\usepackage{changepage}
\usepackage{longtable}
\usepackage{subcaption}
\usepackage{epstopdf}
\usepackage{capt-of}
\usepackage{placeins}

\usepgfplotslibrary{colormaps}
\pgfplotsset{compat=1.18}
\theoremstyle{plain}
\newtheorem{theorem}{Theorem}[section]
\newtheorem{lemma}{Lemma}[section]
\newtheorem{corollary}{Corollary}[section]
\newtheorem{proposition}{Proposition}[section]

\theoremstyle{definition}

\newtheorem{example}{Example}[section]

\theoremstyle{remark}
\newtheorem{remark}{Remark}

\numberwithin{equation}{section}

\newcommand{\Le}{\mathrm{Le}}
\usepackage{amsthm}
\usepackage{placeins}

\providecommand{\Authorcontributions}[2]{\section*{Author contributions}#2}
\providecommand{\Competinginterests}[2]{\section*{Competing interests}#2}

\begin{document}
\date{}
\title[Zeros and Their Distribution for Leonardo Polynomials]{Zeros and Their Distribution for Leonardo Polynomials}
\author[M. Karag\"{o}z]{Muhammet KARAG\"{O}Z}
\address{\.{I}zmir Democracy University, Department of Mathematics\\
35140 Karaba\u{g}lar, \.{I}zmir, T\"{u}rkiye}
\email{muhammet.karagoz@idu.edu.tr}
\author[N. \"{O}zg\"{u}r]{N\.{I}HAL \"{O}ZG\"{U}R}
\address{\.{I}zmir Democracy University, Department of Mathematics\\
35140 Karaba\u{g}lar, \.{I}zmir, T\"{u}rkiye}
\email{nihal.ozgur@idu.edu.tr}
\date{}
\subjclass[2020]{12D10, 30C15, 11B39, 65H04, 47J26}
\keywords{Leonardo polynomials, polynomial zeros, Chebyshev polynomials, fixed-point iteration, Banach fixed-point theorem}

\begin{abstract}
We investigate the zeros of Leonardo polynomials and their distribution in the complex plane. We show that the structure of the zeros differs between even and odd degrees. For even degrees, all zeros are simple and purely imaginary, with imaginary parts lying in \((-2,2)\). For odd degrees, there is exactly one positive real zero, whereas no negative real or purely imaginary zeros occur. We obtain two-sided bounds for the unique positive real zero, prove that it decreases strictly to zero as the degree increases through odd values, and determine its asymptotic behavior.

Using the relation between Leonardo and Fibonacci polynomials together with a complex trigonometric parametrization, we transform the zero equation into a family of fixed-point problems. Suitable invariant disks are constructed, and the corresponding mappings are shown to be contractions. This yields a constructive description of the complete zero set, proves the simplicity of all zeros, and provides geometrically convergent iterations for their numerical approximation. The method also provides explicit control over the approximation error.

Finally, we establish a degree-dependent vertical strip containing all zeros and show that its width tends to zero as the degree increases. Consequently, the zeros become uniformly concentrated near the imaginary axis. Numerical examples illustrate the convergence of the proposed method, the associated error estimates, and the resulting zero distributions.
\end{abstract}

\maketitle

\section{Introduction}
Leonardo polynomials extend the Leonardo number sequence and connect it
with classical polynomial families defined by recurrence relations.
Catarino and Borges~\cite{Catarino and Borges2020} studied the identities and
structural properties of Leonardo numbers; further developments and
generalizations of the sequence were considered in
\cite{AlpKocer2021,Soykan2021}. Metallic generalizations involving the silver, bronze, and copper
cases were investigated by Kumari, Prasad, and Singh~\cite{KumariPrasadSingh2026}. Prasad and Kumari~\cite{PrasadKumari2025} introduced the family of Leonardo polynomials considered in the present paper and established generating functions, summation identities, relations
with Fibonacci, Lucas, and Chebyshev polynomials, and irreducibility
results. Related developments include bivariate Leonardo polynomials
and their connections with Riordan arrays~\cite{AlpKocer2025},
Gaussian and hybrid extensions
\cite{PrasadKumariSingh2026,Yagmur2023},
the $q$-analogues studied by Tu\u{g}lu, Ekinci and \c{C}etin
Keskin~\cite{TugluEkinciKeskin2026}, and recent results on Leonardo polynomials,
including generating functions and linearization formulas,
by Soykan~\cite{Soykan2026}. These studies provide an algebraic basis
for investigating the zeros of Leonardo polynomials and their
distribution in the complex plane.

Zeros of related polynomial families defined by recurrence relations
have also received considerable attention. Hoggatt and Bicknell~\cite{HoggattBicknell1973}
obtained explicit representations for the zeros of the Fibonacci and Lucas
polynomials. The results on Tribonacci polynomials and the bounds of their zeros were given in~\cite{Kaymak2018}, while the zeros of \(R\)-Bonacci polynomials and their
derivatives were studied in~\cite{KaymakOzgur2022}. These works motivate the investigation of root location, multiplicity, and computation for Leonardo polynomials.

A fundamental identity for Leonardo polynomials is
\begin{equation}
\label{eq:leonardo-fibonacci}
\Le_n(x)=2F_{n+1}(x)-1, n\geq0
\end{equation}
where \(F_n(x)\) denotes the Fibonacci polynomial; see
\cite{PrasadKumari2025}. Hence,

\[
\Le_n(x)=0
\quad\Longleftrightarrow\quad
F_{n+1}(x)=\frac12.
\]
Thus, the zeros of $\Le_n$ form the level set
\[
\left\{x\in\mathbb{C}:F_{n+1}(x)=\frac12\right\}.
\]
Consequently, the known formulas for the zeros of $F_{n+1}$ do not directly determine the zeros of $\Le_n$.
In this paper, we study the zeros of \(\Le_n\) from both qualitative and constructive points of view. We first show that their distribution depends essentially on whether \(n\) is even or odd. The even-degree case is characterized by purely imaginary zeros, whereas the odd-degree case exhibits a distinct real zero structure, including a unique positive real zero whose magnitude decreases as \(n\) increases. For every positive odd degree \(n\), the unique positive
zero \(\xi_n\) satisfies
\[
2\sinh\!\left(\frac{\log\varphi}{n+1}\right)
<
\xi_n
\leq
\frac{1}{n+1},
\qquad
\varphi=\frac{1+\sqrt5}{2},
\]
with equality in the upper bound only for \(n=1\).
We also prove that
\[
\xi_{n+2}<\xi_n
\]
and determine the asymptotic behavior
\[
\xi_n\sim\frac{2\log\varphi}{n+1}
\qquad
(n\to\infty,\ n\ \mathrm{odd}).
\]
In particular, the lower bound is asymptotically sharp.
Our main contribution is a constructive fixed-point formulation for
determining the entire zero set. Using the Leonardo--Fibonacci relation and the change of variables
\[
x=2i\cos\theta,
\]
for \(\sin\theta\neq0\), the zero equation is equivalent to
\[
\sin((n+1)\theta)
=
\lambda_n\sin\theta,
\qquad
\lambda_n=\frac{1}{2i^n}.
\]
We associate this equation with \(n\) fixed-point problems, one for each
index \(k=1,\ldots,n\). Here, each index \(k\) determines a distinct
solution branch arising from the inverse-sine representation of the
trigonometric equation. For each branch, an explicit invariant disk is
constructed and the corresponding map is shown to be a contraction.
Banach's fixed-point theorem then gives a unique fixed point in each disk,
geometric convergence from every initial point in the disk, and an explicit
a priori error estimate. An a posteriori bound further relates the difference
between consecutive iterates to the error in the corresponding zero
approximation. The resulting fixed points generate all zeros
of \(\Le_n\) without repetition through the transformation
\[
x_k=2i\cos\theta_*^{(k)},
\qquad k=1,\ldots,n.
\]
Hence every zero is simple, and the fixed-point iterations provide exact
limit representations as well as numerical approximations.

Finally, we derive the localization estimate
\[
|\operatorname{Re}z|
\leq
2\sinh\!\left(\frac{1}{n+1}\right)
\]
for every zero \(z\) of \(\Le_n\). Consequently, the zero set becomes uniformly concentrated around the imaginary axis as the degree increases. Numerical examples illustrate the convergence of the fixed-point iterations, the theoretical error estimates, and the root distributions for representative degrees.
\section{Preliminaries}
\label{sec:preliminaries}

We collect here the identities and standard tools that will be used
throughout the paper. For background on Fibonacci and Lucas numbers,
see~\cite{Koshy2001,Vajda1989}.
The Fibonacci polynomials are defined by
\[
F_0(x)=0,\qquad
F_1(x)=1,\qquad
F_n(x)=xF_{n-1}(x)+F_{n-2}(x),
\quad n\geq2.
\]
Hoggatt and Bicknell~\cite{HoggattBicknell1973} gave the explicit
representation
\begin{equation*}
F_{n+1}(x)
=
\sum_{j=0}^{\lfloor n/2\rfloor}
\binom{n-j}{j}x^{n-2j},
\qquad n\geq0.
\end{equation*}

Let
\[
\alpha(x)=\frac{x+\sqrt{x^2+4}}{2},
\qquad
\beta(x)=\frac{x-\sqrt{x^2+4}}{2}.
\]
In terms of these characteristic roots, the Fibonacci polynomials have the following Binet-type representation~\cite{HoggattBicknell1973}
\[
F_n(x)
=
\frac{\alpha(x)^n-\beta(x)^n}
{\alpha(x)-\beta(x)},
\qquad x^2+4\neq0.
\]
Following Prasad and Kumari~\cite{PrasadKumari2025}, the Leonardo
polynomials are defined by
\[
\Le_0(x)=1,\qquad
\Le_1(x)=2x-1,
\]
and
\[
\Le_n(x)
=
x\Le_{n-1}(x)+\Le_{n-2}(x)+x,
\qquad n\geq2.
\]

For the family of Leonardo polynomials considered here,
the Leonardo--Fibonacci identity
\eqref{eq:leonardo-fibonacci}, established by
Prasad and Kumari~\cite{PrasadKumari2025}, together with
the explicit representation of the Fibonacci polynomials,
yields
\begin{equation}
\label{eq:leonardo-explicit}
\Le_n(x)
=
2\sum_{j=0}^{\lfloor n/2\rfloor}
\binom{n-j}{j}x^{n-2j}-1,
\qquad n\geq0.
\end{equation}

Let \(U_n\) denote the Chebyshev polynomial of the second kind.
We shall use the following standard identities
\cite{MasonHandscomb2002,Rivlin2020}:
\begin{equation}
\label{eq:chebyshev-trigonometric}
U_n(\cos\theta)
=
\frac{\sin((n+1)\theta)}{\sin\theta},
\qquad
\sin\theta\neq0,
\end{equation}
together with
\begin{equation*}
U_n(1)=n+1,
\qquad
U_n(-1)=(-1)^n(n+1),
\end{equation*}
and
\begin{equation*}
U_n(-x)=(-1)^nU_n(x).
\end{equation*}
The Fibonacci--Chebyshev relation
\cite{HoggattBicknell1973,MasonHandscomb2002} is
\begin{equation}
\label{eq:fibonacci-chebyshev}
F_{n+1}(x)
=
i^n
U_n\left(\frac{x}{2i}\right),
\qquad n\geq0.
\end{equation}
Combining \eqref{eq:chebyshev-trigonometric} and
\eqref{eq:fibonacci-chebyshev}, we obtain
\begin{equation*}
F_{n+1}(2i\cos\theta)
=
i^n
\frac{\sin((n+1)\theta)}{\sin\theta},
\qquad
\sin\theta\neq0.
\end{equation*}
For further relations among Fibonacci, Lucas, and Chebyshev
polynomials, see also~\cite{FrontczakGoy2021,SmajlovicSabanacSceta2025}.

\section{Classification of the Zeros}
We begin by determining the real zeros of Leonardo polynomials
and then locate all zeros in the even-degree case.
By~\eqref{eq:leonardo-explicit}, \(\Le_n\) has real coefficients
and degree \(n\) for \(n\geq1\).

For a nonzero polynomial \(P(x)\) with real coefficients,
let \(V(P)\) denote the number of sign changes in its
coefficient sequence, arranged in descending order of degree,
after zero coefficients have been omitted.
Descartes' rule of signs states that the number of positive
real zeros, counted with multiplicity, equals \(V(P)\) or
is smaller by an even integer; see~\cite{Wang2004}.
\begin{theorem}
\label{thm:even-no-real}
For every integer $m\geq1$, the polynomial $\Le_{2m}$ has no real zero.
\end{theorem}

\begin{proof}
By (2.1), for every $x\in\mathbb{R}$,
\[
\operatorname{Le}_{2m}(x)
=1+2\sum_{j=0}^{m-1}
\binom{2m-j}{j}x^{2m-2j}
\geq 1>0,
\]
since all coefficients in the sum are positive and all
exponents are even. Hence $\operatorname{Le}_{2m}$
has no real zero.
\end{proof}

\begin{theorem}
\label{thm:odd-real}
For every integer $m\geq0$, the polynomial $\Le_{2m+1}$ has exactly one real zero. This zero is positive and simple.
\end{theorem}

\begin{proof}
Let \(n=2m+1\) with \(m\geq0\). From~\eqref{eq:leonardo-explicit},
\begin{equation}
\label{eq:classification-odd-expansion}
\Le_{2m+1}(x)
=
2\sum_{j=0}^{m}
\binom{2m+1-j}{j}x^{2m+1-2j}-1.
\end{equation}

All coefficients of the nonconstant terms are positive, whereas the
constant term is \(-1\). Hence, after the zero coefficients are omitted,
the coefficient sequence of \(\Le_{2m+1}(x)\) contains exactly one sign
change. By Descartes' rule of signs, \(\Le_{2m+1}\) has exactly one
positive real zero, counted with multiplicity. Therefore, this zero is
simple.

We next examine the negative real zeros. Replacing \(x\) by \(-x\) in
\eqref{eq:classification-odd-expansion}, and using the fact that all
powers \(2m+1-2j\) are odd, we obtain
\[
\Le_{2m+1}(-x)
=
-2\sum_{j=0}^{m}
\binom{2m+1-j}{j}x^{2m+1-2j}-1.
\]
All nonzero coefficients of \(\Le_{2m+1}(-x)\) are negative, so its
coefficient sequence contains no sign change. Hence, by Descartes'
rule of signs, \(\Le_{2m+1}\) has no negative real zero.

Finally,
\[
\Le_{2m+1}(0)=-1\neq0,
\]
so \(x=0\) is not a zero. Consequently, \(\Le_{2m+1}\) has exactly one
real zero, which is positive and simple.
\end{proof}

\begin{theorem}
\label{thm:positive-root-bounds}
Let $n$ be a positive odd integer, and let $\xi_n$ denote the unique
positive zero of $\Le_n$. If
\[
\varphi=\frac{1+\sqrt{5}}{2},
\]
then
\begin{equation}
\label{eq:classification-positive-bounds}
2\sinh\!\left(\frac{\log\varphi}{n+1}\right)
<
\xi_n
\leq
\frac{1}{n+1}.
\end{equation}
Equality in the upper bound holds if and only if $n=1$. Moreover,
\[
\lim_{m\to\infty}\xi_{2m+1}=0.
\]
\end{theorem}

\begin{proof}
We first prove the upper bound. For odd $n\geq3$, the polynomial
$\Le_n$ can be written as
\[
\Le_n(x)
=
-1+(n+1)x
+
\sum_{\substack{3\leq k\leq n\\ k\ \mathrm{odd}}}
a_kx^k,
\qquad a_k>0.
\]
Since $\Le_n(\xi_n)=0$ and $\xi_n>0$, we have
\[
1
=
(n+1)\xi_n
+
\sum_{\substack{3\leq k\leq n\\ k\ \mathrm{odd}}}
a_k\xi_n^k
>
(n+1)\xi_n.
\]
Thus
\[
\xi_n<\frac{1}{n+1}.
\]
For $n=1$, $\Le_1(x)=2x-1$, so
\[
\xi_1=\frac12=\frac{1}{1+1}.
\]
Hence
\[
\xi_n\leq\frac{1}{n+1},
\]
with equality if and only if $n=1$.

For the lower bound, using
\[
\Le_n(x)=2F_{n+1}(x)-1,
\]
the equation $\Le_n(\xi_n)=0$ gives
\begin{equation}
\label{eq:positive-root-fibonacci-value}
F_{n+1}(\xi_n)=\frac12.
\end{equation}
Since $\xi_n>0$, write
\[
\xi_n=2\sinh t
\]
for some $t>0$. Then
\[
\sqrt{\xi_n^2+4}=2\cosh t,
\]
and the characteristic roots associated with the Fibonacci polynomial
recurrence are
\[
\alpha=e^t,
\qquad
\beta=-e^{-t}.
\]
Since \(n+1\) is even, the Binet representation given above yields
\[
F_{n+1}(\xi_n)
=
\frac{e^{(n+1)t}-e^{-(n+1)t}}
     {e^t+e^{-t}}
=
\frac{\sinh((n+1)t)}{\cosh t}.
\]
Together with \eqref{eq:positive-root-fibonacci-value}, this gives
\[
\sinh((n+1)t)=\frac12\cosh t>\frac12.
\]
Since
\[
\sinh(\log\varphi)
=
\frac{\varphi-\varphi^{-1}}{2}
=
\frac12,
\]
and $\sinh$ is strictly increasing on $\mathbb{R}$, we obtain
\[
(n+1)t>\log\varphi.
\]
Therefore,
\[
t>\frac{\log\varphi}{n+1},
\]
and hence
\[
\xi_n
=
2\sinh t
>
2\sinh\!\left(\frac{\log\varphi}{n+1}\right).
\]
This proves \eqref{eq:classification-positive-bounds}.

Finally, as $n\to\infty$ through odd integers,
\[
2\sinh\!\left(\frac{\log\varphi}{n+1}\right)\to0
\qquad\text{and}\qquad
\frac{1}{n+1}\to0.
\]
Thus, by the squeeze theorem,
\[
\lim_{\substack{n\to\infty\\ n\ \mathrm{odd}}}\xi_n=0,
\]
or equivalently,
\[
\lim_{m\to\infty}\xi_{2m+1}=0.
\]
\end{proof}
\begin{corollary}
\label{cor:positive-root-asymptotics}
For each positive odd integer \(n\), let \(\xi_n\) denote
the unique positive zero of \(\Le_n\), and let \(\varphi=\frac{1+\sqrt5}{2}\).
Then
\[
\xi_{n+2}<\xi_n
\]
for every positive odd integer \(n\). Moreover,
\begin{equation}
\label{eq:positive-root-asymptotic}
\lim_{\substack{n\to\infty\\ n\ \mathrm{odd}}}
(n+1)\xi_n
=
2\log\varphi.
\end{equation}
Consequently,
\[
\xi_n\sim\frac{2\log\varphi}{n+1}
\qquad
(n\to\infty,\ n\ \mathrm{odd}),
\]
and the lower bound in
\eqref{eq:classification-positive-bounds}
is asymptotically sharp.
\end{corollary}

\begin{proof}
Let \(n\) be a positive odd integer.
The Fibonacci recurrence and
\eqref{eq:leonardo-fibonacci} give
\[
\begin{aligned}
\Le_{n+2}(x)-\Le_n(x)
&=
2\bigl(F_{n+3}(x)-F_{n+1}(x)\bigr)\\
&=
2xF_{n+2}(x).
\end{aligned}
\]
Since \(F_{n+2}(x)>0\) for \(x>0\), evaluating this identity
at \(x=\xi_n\) yields
\[
\Le_{n+2}(\xi_n)
=
2\xi_nF_{n+2}(\xi_n)
>
0.
\]
On the other hand, \(\Le_{n+2}(0)=-1\).
The Intermediate Value Theorem therefore gives a zero of
\(\Le_{n+2}\) in \((0,\xi_n)\).
By the uniqueness of the positive zero established in
Theorem~\ref{thm:odd-real}, this zero is \(\xi_{n+2}\).
Hence
\[
0<\xi_{n+2}<\xi_n.
\]

To obtain the asymptotic behavior, write
\[
\xi_n=2\sinh t_n,
\qquad t_n>0.
\]
By Theorem~\ref{thm:positive-root-bounds},
\(\xi_n\to0\) as \(n\to\infty\) through odd integers,
and therefore \(t_n\to0\).
As shown in the proof of that theorem,
\[
\sinh\bigl((n+1)t_n\bigr)
=
\frac12\cosh t_n.
\]
Consequently,
\[
\sinh\bigl((n+1)t_n\bigr)
\longrightarrow
\frac12
=
\sinh(\log\varphi).
\]
The continuity of the inverse hyperbolic sine gives
\[
(n+1)t_n\longrightarrow\log\varphi.
\]
Using \(\sinh t_n/t_n\to1\), we obtain
\[
(n+1)\xi_n
=
2(n+1)t_n\,\frac{\sinh t_n}{t_n}
\longrightarrow
2\log\varphi.
\]
This proves \eqref{eq:positive-root-asymptotic}.

Finally,
\[
2(n+1)\sinh\!\left(
\frac{\log\varphi}{n+1}
\right)
\longrightarrow
2\log\varphi.
\]
Combining this limit with
\eqref{eq:positive-root-asymptotic}, we conclude that
\[
\lim_{\substack{n\to\infty\\ n\ \mathrm{odd}}}
\frac{\xi_n}{
2\sinh\!\left(\dfrac{\log\varphi}{n+1}\right)
}
=
1.
\]
Thus the lower bound is asymptotically sharp.
\end{proof}
\begin{theorem}
\label{thm:odd-no-imaginary}
For every integer $m\geq0$, the polynomial $\Le_{2m+1}$ has no purely imaginary zero.
\end{theorem}

\begin{proof}
Let $y\in\mathbb{R}$. Substituting $x=iy$ into~\eqref{eq:classification-odd-expansion} gives
\[
\Le_{2m+1}(iy)
=
-1+2i\sum_{j=0}^{m}
(-1)^{m-j}
\binom{2m+1-j}{j}
y^{2m+1-2j}.
\]
The sum is real, and therefore
\[
\operatorname{Re}\bigl(\Le_{2m+1}(iy)\bigr)=-1.
\]
Thus $\Le_{2m+1}(iy)\neq0$ for every real $y$.
\end{proof}

To locate the zeros in the even-degree case, we use the known Leonardo--Chebyshev identity~\cite{PrasadKumari2025}
\begin{equation}
\label{eq:classification-chebyshev}
\Le_n(x)
=
2i^nU_n\!\left(\frac{x}{2i}\right)-1,
\qquad n\geq0,
\end{equation}
where $U_n$ is the Chebyshev polynomial of the second kind. This follows by combining $\Le_n(x)=2F_{n+1}(x)-1$ with
$F_{n+1}(x)=i^nU_n(x/(2i))$.

\begin{theorem}
\label{thm:even-all-imaginary}
For every integer \(m\geq1\), all \(2m\) zeros of\hspace{0.2em}\(\Le_{2m}\) are
distinct and purely imaginary, with imaginary parts lying in
\((-2,2)\).
More precisely, there exist real numbers
\[
0<\mu_1<\mu_2<\cdots<\mu_m<2
\]
such that the zero set of $\Le_{2m}$ is
\[
\{\pm i\mu_1,\ldots,\pm i\mu_m\}.
\]
In particular, every zero of $\Le_{2m}$ is simple.
\end{theorem}

\begin{proof}
By~\eqref{eq:classification-chebyshev} and
\(i^{2m}=(-1)^m\), we have
\[
\Le_{2m}(x)
=
2(-1)^mU_{2m}\left(\frac{x}{2i}\right)-1.
\]
Thus, setting \(t=x/(2i)\), the equation \(\Le_{2m}(x)=0\)
is equivalent to
\[
U_{2m}(t)=\frac{(-1)^m}{2}.
\]
Define
\[
H_m(t)=U_{2m}(t)-\frac{(-1)^m}{2}.
\]
The zeros of \(H_m\) correspond to those of \(\Le_{2m}\)
under the transformation \(x=2it\).

For \(j=0,1,\ldots,2m\), let
\[
\theta_j=\frac{(j+\frac12)\pi}{2m+1},
\qquad
t_j=\cos\theta_j.
\]
Since \(0<\theta_0<\theta_1<\cdots<\theta_{2m}<\pi\)
and cosine is strictly decreasing on \((0,\pi)\), we have
\[
1>t_0>t_1>\cdots>t_{2m}>-1.
\]
Using the trigonometric representation of \(U_{2m}\) and
the equality \((2m+1)\theta_j=(j+\frac12)\pi\), we obtain
\[
U_{2m}(t_j)
=
\frac{\sin((2m+1)\theta_j)}{\sin\theta_j}
=
\frac{(-1)^j}{\sin\theta_j}.
\]
Since \(0<\theta_j<\pi\), we have \(0<\sin\theta_j\leq1\).
Consequently,
\[
|U_{2m}(t_j)|\geq1,
\qquad
\operatorname{sgn}(U_{2m}(t_j))=(-1)^j.
\]

The constant subtracted in the definition of \(H_m\) has
absolute value \(1/2\), which is smaller than
\(|U_{2m}(t_j)|\). Therefore, this subtraction does not
change the sign, and
\[
\operatorname{sgn}(H_m(t_j))=(-1)^j,
\qquad j=0,\ldots,2m.
\]
In particular,
\[
H_m(t_j)H_m(t_{j+1})<0,
\qquad j=0,\ldots,2m-1.
\]
By the Intermediate Value
Theorem~\cite[Theorem~5.3.7]{Bartle and Sherbert2011},
\(H_m\) has at least one real zero in each interval
\((t_{j+1},t_j)\), \(j=0,\ldots,2m-1\).
These \(2m\) intervals are pairwise disjoint and contained
in \((-1,1)\), so \(H_m\) has at least \(2m\) distinct
real zeros in \((-1,1)\).

Since \(H_m\) has degree \(2m\), there is exactly one zero
in each of these intervals, and these zeros exhaust its
zero set. Moreover, all of them are simple, since a
degree-\(2m\) polynomial with \(2m\) distinct zeros cannot
have a multiple zero.

The polynomial \(U_{2m}\) is even, so
\(H_m(-t)=H_m(t)\). Hence the zeros of \(H_m\) occur in
symmetric pairs. Furthermore, \(U_{2m}(0)=(-1)^m\) gives
\[
H_m(0)
=
(-1)^m-\frac{(-1)^m}{2}
=
\frac{(-1)^m}{2}
\neq0.
\]
Thus zero is not a root, and the zeros of \(H_m\) can be
written as
\[
\pm s_1,\ldots,\pm s_m,
\qquad
0<s_1<\cdots<s_m<1.
\]

Returning to \(x=2it\), the zeros of \(\Le_{2m}\) are
\(\pm2is_1,\ldots,\pm2is_m\).
Setting \(\mu_j=2s_j\), we obtain
\[
\operatorname{Zeros}(\Le_{2m})
=
\{\pm i\mu_1,\ldots,\pm i\mu_m\},
\qquad
0<\mu_1<\cdots<\mu_m<2.
\]
Therefore, all \(2m\) zeros of \(\Le_{2m}\) are distinct,
purely imaginary, with imaginary parts lying in \((-2,2)\).
Since \(\Le_{2m}\) has degree \(2m\), every zero is simple.
\end{proof}
\begin{corollary}
\label{cor:odd-remaining-zeros}
For every integer \(m\geq1\), the zeros of \(\Le_{2m+1}\), other than its
unique positive real zero, consist of \(2m\) nonreal zeros, counted with
multiplicity. These zeros occur in complex conjugate pairs, and each of
them has both nonzero real and nonzero imaginary parts.
\end{corollary}

\begin{proof}
By Theorem~\ref{thm:odd-real}, the polynomial \(\Le_{2m+1}\)
has exactly one real zero, which is positive and simple.
Since its degree is \(2m+1\), the Fundamental Theorem of Algebra
implies that the remaining \(2m\) zeros are nonreal, counted
with multiplicity.

These zeros occur in complex conjugate pairs because
\(\Le_{2m+1}\) has real coefficients. Their imaginary parts
are nonzero since they are nonreal, and their real parts
are nonzero by Theorem~\ref{thm:odd-no-imaginary}.
\end{proof}
\section{A Fixed-Point Method for All Zeros}
\label{sec:fixed-point}

We use the principal branches of the inverse sine and square root;
for background on complex elementary functions,
see~\cite{Brown and Churchill2014}.
We construct all zeros of \(\Le_n\) using fixed-point
iterations based on the parametrization
\(x=2i\cos\theta\). Set
\[
\lambda_n=\frac{1}{2i^n},
\qquad
|\lambda_n|=\frac12.
\]
For each \(k\in\mathbb Z\), define
\begin{equation}
\label{eq:fp-map}
T_{n,k}(\theta)
=
\frac{
k\pi+(-1)^k\arcsin(\lambda_n\sin\theta)
}{n+1},
\end{equation}
whenever \(|\lambda_n\sin\theta|<1\).
The index \(k\) labels the solution families arising from the sine equation.

\begin{lemma}
\label{lem:fp-equivalence}
The points \(x=\pm2i\) are not zeros of \(\Le_n\).
For every \(\theta\in\mathbb C\) with \(\sin\theta\neq0\),
the equation \(\Le_n(2i\cos\theta)=0\) is equivalent to
\begin{equation}
\label{eq:fp-trig}
\sin((n+1)\theta)=\lambda_n\sin\theta.
\end{equation}
If, in addition, \(|\lambda_n\sin\theta|<1\), these
conditions are equivalent to
\[
\theta=T_{n,k}(\theta)
\]
for some \(k\in\mathbb Z\).
\end{lemma}

\begin{proof}
We first exclude \(x=\pm2i\), which correspond to
\(\sin\theta=0\) under \(x=2i\cos\theta\).
Using \(U_n(1)=n+1\) and \(U_n(-1)=(-1)^n(n+1)\)
in~\eqref{eq:classification-chebyshev}, we obtain
\[
\Le_n(\pm2i)
=
2(\pm i)^n(n+1)-1
\neq0,
\]
since \(\left|2(\pm i)^n(n+1)\right|=2(n+1)>1\).

Now let \(\theta\in\mathbb C\) satisfy \(\sin\theta\neq0\).
Combining~\eqref{eq:classification-chebyshev}
and~\eqref{eq:chebyshev-trigonometric} gives
\begin{equation}
\label{eq:fp-trigonometric-representation}
\Le_n(2i\cos\theta)
=
2i^n\frac{\sin((n+1)\theta)}{\sin\theta}-1.
\end{equation}
Since \(\sin\theta\neq0\), setting this expression equal
to zero gives~\eqref{eq:fp-trig}, and the converse follows
by substitution.

Assume further that \(|\lambda_n\sin\theta|<1\), and set
\(w=\lambda_n\sin\theta\) and \(\alpha=\arcsin w\).
Then \(\sin\alpha=w\), so the equation \(\sin z=w\)
is equivalent to
\[
0=\sin z-\sin\alpha
=
2\cos\left(\frac{z+\alpha}{2}\right)
\sin\left(\frac{z-\alpha}{2}\right).
\]
Thus either \(\sin((z-\alpha)/2)=0\) or
\(\cos((z+\alpha)/2)=0\). The first case gives
\((z-\alpha)/2=j\pi\), while the second gives
\((z+\alpha)/2=\pi/2+j\pi\). Hence all solutions are
\[
z=2j\pi+\alpha
\quad\text{or}\quad
z=(2j+1)\pi-\alpha,
\qquad j\in\mathbb Z.
\]
These two families can be written in the single form
\[
z=k\pi+(-1)^k\arcsin w,
\qquad k\in\mathbb Z.
\]

If~\eqref{eq:fp-trig} holds, applying this representation
with \(z=(n+1)\theta\) and \(w=\lambda_n\sin\theta\) gives
\[
(n+1)\theta
=
k\pi+(-1)^k\arcsin(\lambda_n\sin\theta)
\]
for some \(k\in\mathbb Z\). Dividing by \(n+1\) yields
\(\theta=T_{n,k}(\theta)\).

Conversely, if \(\theta=T_{n,k}(\theta)\), taking the sine
of both sides of
\[
(n+1)\theta
=
k\pi+(-1)^k\arcsin(\lambda_n\sin\theta)
\]
and using \(\sin(k\pi+(-1)^k\alpha)=\sin\alpha\)
recovers~\eqref{eq:fp-trig}. This completes the proof.
\end{proof}
\begin{theorem}
\label{thm:fp-contraction}
For $k=1,\ldots,n$, let
\[
D_{n,k}
=
\left\{
\theta\in\mathbb{C}:
\left|\theta-\frac{k\pi}{n+1}\right|
\leq\frac{1}{n+1}
\right\}.
\]
Then $T_{n,k}$ is well defined on $D_{n,k}$ and maps this disk into itself. Moreover,
\begin{equation}
\label{eq:fp-lipschitz}
|T_{n,k}(\theta)-T_{n,k}(\eta)|
\leq
\kappa_n|\theta-\eta|,
\qquad \theta,\eta\in D_{n,k},
\end{equation}
where
\[
\kappa_n
=
\frac{\sqrt{61}}{8(n+1)}
<1.
\]
\end{theorem}

\begin{proof}
Fix $k\in\{1,\ldots,n\}$ and write $\theta=u+iv\in D_{n,k}$, where $u,v\in\mathbb{R}$. Since the center of the disk is real,
\[
|v|\leq\frac{1}{n+1}\leq\frac12.
\]
The complex sine satisfies
\[
|\sin(u+iv)|^2
=
\sin^2u+\sinh^2v
\leq
\cosh^2v.
\]
Hence
\[
|\lambda_n\sin\theta|
\leq
\frac12\cosh\!\left(\frac{1}{n+1}\right)
\leq
\frac12\cosh\!\left(\frac12\right).
\]
Also,
\[
\cosh\!\left(\frac12\right)
=
1+\sum_{j=1}^{\infty}\frac{1}{4^j(2j)!}
\leq
1+\frac12\sum_{j=1}^{\infty}\frac{1}{4^j}
=
\frac76.
\]
It follows that
\begin{equation}
\label{eq:fp-argument-bound}
|\lambda_n\sin\theta|
\leq
\frac{7}{12}
<
\frac35
<
1.
\end{equation}
Thus, the principal inverse sine in~\eqref{eq:fp-map} is well defined throughout $D_{n,k}$. We use the common bound $\rho=3/5$ below.

For $|w|<1$, integration along the segment from $0$ to $w$ gives
\[
\arcsin w
=
w\int_0^1
\frac{dt}{\sqrt{1-t^2w^2}}.
\]
Since
\[
\left|\sqrt{1-t^2w^2}\right|
\geq
\sqrt{1-t^2|w|^2},
\]
we obtain
\[
|\arcsin w|
\leq
|w|\int_0^1
\frac{dt}{\sqrt{1-t^2|w|^2}}
\leq
\frac{|w|}{\sqrt{1-|w|^2}}.
\]
In particular, $|w|\leq3/5$ implies $|\arcsin w|\leq3/4$. Therefore,
\[
\left|
T_{n,k}(\theta)-\frac{k\pi}{n+1}
\right|
\leq
\frac{3}{4(n+1)}
<
\frac{1}{n+1},
\qquad \theta\in D_{n,k},
\]
which proves $T_{n,k}(D_{n,k})\subseteq D_{n,k}$.

The strict bound in~\eqref{eq:fp-argument-bound} also shows that $T_{n,k}$ is analytic on an open set containing $D_{n,k}$. Differentiating gives
\begin{equation}
\label{eq:fp-derivative}
T_{n,k}'(\theta)
=
\frac{(-1)^k}{n+1}
\frac{\lambda_n\cos\theta}
{\sqrt{1-\lambda_n^2\sin^2\theta}}.
\end{equation}
From $|\lambda_n\sin\theta|\leq\rho$ and $|\lambda_n|=1/2$, we have $|\sin\theta|\leq2\rho$. Consequently,
\[
|\cos\theta|^2
=
|1-\sin^2\theta|
\leq
1+|\sin\theta|^2
\leq
1+4\rho^2.
\]
Furthermore,
\[
\left|
\sqrt{1-\lambda_n^2\sin^2\theta}
\right|
\geq
\sqrt{1-|\lambda_n\sin\theta|^2}
\geq
\sqrt{1-\rho^2}.
\]
Applying these estimates to~\eqref{eq:fp-derivative}, with $\rho=3/5$, gives
\[
|T_{n,k}'(\theta)|
\leq
\frac{\sqrt{1+4\rho^2}}
{2(n+1)\sqrt{1-\rho^2}}
=
\frac{\sqrt{61}}{8(n+1)}
=
\kappa_n.
\]
Since $n\geq1$, we have $\kappa_n\leq\sqrt{61}/16<1$.

The disk $D_{n,k}$ is convex. Hence, for any $\theta,\eta\in D_{n,k}$, the segment joining them remains in $D_{n,k}$, and
\[
T_{n,k}(\theta)-T_{n,k}(\eta)
=
(\theta-\eta)
\int_0^1
T_{n,k}'\bigl(\eta+t(\theta-\eta)\bigr)\,dt.
\]
Taking absolute values and applying the derivative bound proves~\eqref{eq:fp-lipschitz}.
\end{proof}

\begin{theorem}
\label{thm:fp-convergence}
For every \(k=1,\ldots,n\), the map \(T_{n,k}\) has a unique
fixed point \(\theta_*^{(k)}\) in \(D_{n,k}\).
For every initial value \(\theta_0^{(k)}\in D_{n,k}\), the iteration
\begin{equation}
\label{eq:fp-iteration}
\theta_{r+1}^{(k)}
=
T_{n,k}\bigl(\theta_r^{(k)}\bigr),
\qquad r=0,1,2,\ldots,
\end{equation}
converges geometrically to \(\theta_*^{(k)}\) and satisfies
\begin{equation}
\label{eq:fp-error}
\left|
\theta_r^{(k)}-\theta_*^{(k)}
\right|
\leq
\frac{\kappa_n^r}{1-\kappa_n}
\left|
\theta_1^{(k)}-\theta_0^{(k)}
\right|,
\qquad r\geq0.
\end{equation}
In particular, the initial values
\[
\theta_0^{(k)}=\frac{k\pi}{n+1},
\qquad k=1,\ldots,n,
\]
are admissible for every \(n\geq1\).
\end{theorem}

\begin{proof}
Each $D_{n,k}$ is a nonempty closed subset of $\mathbb{C}$ and is therefore complete in the usual metric. Theorem~\ref{thm:fp-contraction} shows that $T_{n,k}$ maps $D_{n,k}$ into itself and is a contraction with constant $\kappa_n<1$. Banach's fixed-point theorem~\cite{Berinde2007} gives the unique fixed point and convergence from every initial value in the disk.

Repeated application of~\eqref{eq:fp-lipschitz} yields
\[
\left|
\theta_{j+1}^{(k)}-\theta_j^{(k)}
\right|
\leq
\kappa_n^j
\left|
\theta_1^{(k)}-\theta_0^{(k)}
\right|,
\qquad j\geq0.
\]
Thus,
\[
\begin{aligned}
\left|
\theta_r^{(k)}-\theta_*^{(k)}
\right|
&\leq
\sum_{j=r}^{\infty}
\left|
\theta_{j+1}^{(k)}-\theta_j^{(k)}
\right|\\
&\leq
\sum_{j=r}^{\infty}
\kappa_n^j
\left|
\theta_1^{(k)}-\theta_0^{(k)}
\right|\\
&=
\frac{\kappa_n^r}{1-\kappa_n}
\left|
\theta_1^{(k)}-\theta_0^{(k)}
\right|.
\end{aligned}
\]
This proves~\eqref{eq:fp-error} and geometric convergence. The prescribed initial values are the centers of the disks and hence belong to them.
\end{proof}
\begin{theorem}
\label{thm:fp-all-zeros}
Let \(n\geq1\). For each \(k=1,\ldots,n\), let
\(\theta_*^{(k)}\) denote the unique fixed point of \(T_{n,k}\)
in \(D_{n,k}\), given by Theorem~\ref{thm:fp-convergence}.
Then the numbers
\[
x_k=2i\cos\theta_*^{(k)},
\qquad k=1,\ldots,n,
\]
are pairwise distinct and constitute all zeros of \(\Le_n\).
In particular, every zero of \(\Le_n\) is simple.

Moreover, for any initial values
\(\theta_0^{(k)}\in D_{n,k}\), the iterations defined
by~\eqref{eq:fp-iteration} yield the exact limit representations
\[
x_k
=
\lim_{r\to\infty}2i\cos\theta_r^{(k)},
\qquad k=1,\ldots,n.
\]
\end{theorem}
\begin{proof}
For each \(k=1,\ldots,n\), the inclusion
\(\theta_*^{(k)}\in D_{n,k}\) gives
\[
0<\frac{k\pi-1}{n+1}
\leq \operatorname{Re}\theta_*^{(k)}
\leq \frac{k\pi+1}{n+1}<\pi.
\]
Thus all fixed points belong to
\[
S=\{\theta\in\mathbb C:0<\operatorname{Re}\theta<\pi\},
\]
and \(\sin\theta_*^{(k)}\neq0\).
By Theorem~\ref{thm:fp-contraction},
\(\lvert\lambda_n\sin\theta_*^{(k)}\rvert<1\).
Hence Lemma~\ref{lem:fp-equivalence} implies
\[
\Le_n\bigl(2i\cos\theta_*^{(k)}\bigr)=0.
\]

The disks \(D_{n,k}\) are pairwise disjoint, since the
distance between any two distinct centers is at least
\(\pi/(n+1)>2/(n+1)\), the sum of their radii.
Therefore, their fixed points are pairwise distinct.

Moreover, cosine is injective on \(S\).
Indeed, if \(\theta,\eta\in S\) and
\(\cos\theta=\cos\eta\), then
\[
\theta-\eta=2j\pi
\quad\text{or}\quad
\theta+\eta=2j\pi
\qquad (j\in\mathbb Z).
\]
The second possibility is excluded by
\(0<\operatorname{Re}(\theta+\eta)<2\pi\),
while the first forces \(j=0\), since
\(-\pi<\operatorname{Re}(\theta-\eta)<\pi\).
Thus \(\theta=\eta\).

Consequently, the numbers
\(x_k=2i\cos\theta_*^{(k)}\) are \(n\) distinct zeros
of the degree-\(n\) polynomial \(\Le_n\).
They therefore constitute its complete zero set,
and every zero is simple.

Finally, the limit representations follow from
Theorem~\ref{thm:fp-convergence} and the continuity
of the cosine function.
\end{proof}
\begin{remark}
The simplicity of the zeros of $\Le_n$ also follows from
its irreducibility over $\mathbb{Q}$, established by
Prasad and Kumari~\cite{PrasadKumari2025}, since
$\mathbb{Q}$ has characteristic zero.
The argument above provides an independent constructive
proof by obtaining all $n$ zeros as pairwise distinct
limits of fixed-point iterations on explicit invariant disks.
\end{remark}
\begin{proposition}
\label{prop:fp-symmetry}
Let \(n\geq1\), and let
\[
\theta_*^{(k)}\in D_{n,k},
\qquad
x_k=2i\cos\theta_*^{(k)},
\qquad k=1,\ldots,n,
\]
be the fixed points and zeros given by
Theorem~\ref{thm:fp-all-zeros}. Then
\begin{equation}
\label{eq:fp-branch-symmetry}
\theta_*^{(n+1-k)}
=
\pi-\overline{\theta_*^{(k)}},
\qquad
x_{n+1-k}
=
\overline{x_k},
\qquad k=1,\ldots,n.
\end{equation}
In particular, if \(n\) is odd, the branch corresponding to
\[
k=\frac{n+1}{2}
\]
yields the unique positive real zero:
\[
x_{(n+1)/2}=\xi_n.
\]
If \(n\) is even, then
\[
x_{n+1-k}=-x_k,
\qquad k=1,\ldots,n.
\]
\end{proposition}

\begin{proof}
Fix \(k\in\{1,\ldots,n\}\), set \(j=n+1-k\), and define
\[
\mathcal{R}(\theta)=\pi-\overline{\theta}.
\]
Since
\[
\left|
\mathcal{R}(\theta)-\frac{j\pi}{n+1}
\right|
=
\left|
\theta-\frac{k\pi}{n+1}
\right|,
\]
we have
\[
\mathcal{R}(D_{n,k})=D_{n,j}.
\]

On the unit disk, the principal inverse sine is odd and
commutes with complex conjugation:
\[
\arcsin(-w)=-\arcsin w,
\qquad
\arcsin(\overline{w})
=
\overline{\arcsin w},
\qquad |w|<1.
\]
Moreover,
\[
\sin(\pi-\overline{\theta})
=
\overline{\sin\theta},
\qquad
\overline{\lambda_n}=(-1)^n\lambda_n.
\]
The bound established in
Theorem~\ref{thm:fp-contraction} therefore gives, for every
\(\theta\in D_{n,k}\),
\[
\arcsin\!\left(
\lambda_n\sin(\pi-\overline{\theta})
\right)
=
(-1)^n
\overline{
\arcsin(\lambda_n\sin\theta)
}.
\]
Since \(j=n+1-k\) and
\((-1)^{j+n}=-(-1)^k\), it follows that
\begin{equation}
\label{eq:fp-map-symmetry}
T_{n,j}(\pi-\overline{\theta})
=
\pi-\overline{T_{n,k}(\theta)}.
\end{equation}

Applying~\eqref{eq:fp-map-symmetry} to the fixed point
\(\theta_*^{(k)}\), we obtain
\[
T_{n,j}\!\left(
\pi-\overline{\theta_*^{(k)}}
\right)
=
\pi-\overline{\theta_*^{(k)}}.
\]
Thus \(\pi-\overline{\theta_*^{(k)}}\) is a fixed point of
\(T_{n,j}\) in \(D_{n,j}\). By the uniqueness established in
Theorem~\ref{thm:fp-convergence},
\[
\theta_*^{(j)}
=
\pi-\overline{\theta_*^{(k)}}.
\]
Consequently,
\[
\begin{aligned}
x_j
&=
2i\cos\!\left(
\pi-\overline{\theta_*^{(k)}}
\right)\\
&=
-2i\,\overline{\cos\theta_*^{(k)}}\\
&=
\overline{x_k}.
\end{aligned}
\]
This proves~\eqref{eq:fp-branch-symmetry}.

If \(n\) is odd, taking \(k=(n+1)/2\) gives
\[
x_{(n+1)/2}
=
\overline{x_{(n+1)/2}},
\]
so this zero is real. By Theorem~\ref{thm:odd-real}, it is
the unique positive real zero \(\xi_n\).

If \(n\) is even, all zeros are purely imaginary by
Theorem~\ref{thm:even-all-imaginary}. Hence
\(\overline{x_k}=-x_k\), which yields
\(x_{n+1-k}=-x_k\).
\end{proof}
\begin{remark}
\label{rem:fp-symmetry-computation}
Proposition~\ref{prop:fp-symmetry} explicitly identifies the paired
branches \(k\) and \(n+1-k\). For the central initial values, this
pairing is also preserved throughout the iteration. Indeed, for
\[
\theta_0^{(k)}=\frac{k\pi}{n+1},
\qquad k=1,\ldots,n,
\]
we have
\[
\theta_0^{(n+1-k)}
=
\frac{(n+1-k)\pi}{n+1}
=
\pi-\frac{k\pi}{n+1}
=
\pi-\overline{\theta_0^{(k)}},
\]
since the initial values are real.

Now suppose that, for some \(r\geq0\),
\[
\theta_r^{(n+1-k)}
=
\pi-\overline{\theta_r^{(k)}}.
\]
Using the iteration formula and
\eqref{eq:fp-map-symmetry}, we obtain
\[
\begin{aligned}
\theta_{r+1}^{(n+1-k)}
&=
T_{n,n+1-k}\!\left(\theta_r^{(n+1-k)}\right)\\
&=
T_{n,n+1-k}\!\left(
\pi-\overline{\theta_r^{(k)}}
\right)\\
&=
\pi-\overline{
T_{n,k}\!\left(\theta_r^{(k)}\right)
}\\
&=
\pi-\overline{\theta_{r+1}^{(k)}}.
\end{aligned}
\]
Hence, by induction,
\[
\theta_r^{(n+1-k)}
=
\pi-\overline{\theta_r^{(k)}},
\qquad r\geq0.
\]

Therefore, the symmetry between the paired branches is preserved at
every step of the fixed-point iteration. Consequently, it is sufficient
to compute only the branches
\[
k=1,\ldots,
\left\lceil\frac{n}{2}\right\rceil,
\]
since the remaining zeros can be recovered by complex conjugation.
\end{remark}

\begin{theorem}
\label{thm:fp-strip}
For every $n\geq1$, all zeros $z$ of $\Le_n$ satisfy
\begin{equation}
\label{eq:fp-strip}
|\operatorname{Re}z|
\leq
2\sinh\!\left(\frac{1}{n+1}\right).
\end{equation}
Consequently,
\[
\lim_{n\to\infty}
\max_{\Le_n(z)=0}
|\operatorname{Re}z|
=
0.
\]
\end{theorem}

\begin{proof}
By Theorem~\ref{thm:fp-all-zeros}, each zero can be written as $z=2i\cos\theta_*^{(k)}$ for some $k\in\{1,\ldots,n\}$, with $\theta_*^{(k)}\in D_{n,k}$. Write $\theta_*^{(k)}=u+iv$, where $u,v\in\mathbb{R}$. Since the disk is centered on the real axis,
\[
|v|\leq\frac{1}{n+1}.
\]
The identity
\[
z
=
2i\cos(u+iv)
=
2\sin u\sinh v
+
2i\cos u\cosh v
\]
then gives
\[
|\operatorname{Re}z|
=
2|\sin u|\,|\sinh v|
\leq
2\sinh\!\left(\frac{1}{n+1}\right).
\]
This proves~\eqref{eq:fp-strip}. Since the upper bound tends to zero as $n\to\infty$, the same holds for the maximum distance of the zeros from the imaginary axis.
\end{proof}
\begin{corollary}
\label{cor:fp-aposteriori}
Let \(n\geq1\), \(k\in\{1,\ldots,n\}\), and let
\(\theta_r^{(k)}\) be the iteration defined by
\eqref{eq:fp-iteration}, with
\(\theta_0^{(k)}\in D_{n,k}\).
Let
\[
x_k=2i\cos\theta_*^{(k)}
\]
be the corresponding zero of \(\Le_n\).
For \(r\geq0\), set
\[
\widehat{\theta}_k=\theta_{r+1}^{(k)},
\qquad
\widehat{x}_k=2i\cos\widehat{\theta}_k.
\]
Then
\begin{equation}
\label{eq:fp-aposteriori-angle}
\left|
\widehat{\theta}_k-\theta_*^{(k)}
\right|
\leq
\frac{\kappa_n}{1-\kappa_n}
\left|
\theta_{r+1}^{(k)}-\theta_r^{(k)}
\right|,
\end{equation}
and
\begin{equation}
\label{eq:fp-aposteriori-root}
\left|\widehat{x}_k-x_k\right|
\leq
2\cosh\!\left(\frac{1}{n+1}\right)
\frac{\kappa_n}{1-\kappa_n}
\left|
\theta_{r+1}^{(k)}-\theta_r^{(k)}
\right|.
\end{equation}
Consequently, for any \(\varepsilon>0\), the stopping condition
\[
\left|
\theta_{r+1}^{(k)}-\theta_r^{(k)}
\right|<\varepsilon
\]
implies
\begin{equation}
\label{eq:fp-stopping-root-error}
\left|\widehat{x}_k-x_k\right|
<
2\cosh\!\left(\frac{1}{n+1}\right)
\frac{\kappa_n}{1-\kappa_n}\,\varepsilon.
\end{equation}
\end{corollary}

\begin{proof}
By the contraction property and the triangle inequality,
\[
\begin{aligned}
\left|
\widehat{\theta}_k-\theta_*^{(k)}
\right|
&\leq
\kappa_n
\left|
\theta_r^{(k)}-\theta_*^{(k)}
\right|\\
&\leq
\kappa_n
\left(
\left|
\theta_{r+1}^{(k)}-\theta_r^{(k)}
\right|
+
\left|
\widehat{\theta}_k-\theta_*^{(k)}
\right|
\right).
\end{aligned}
\]
Since \(\kappa_n<1\), rearranging gives
\eqref{eq:fp-aposteriori-angle}.

To estimate the error in the corresponding zero, define
\[
g(\theta)=2i\cos\theta.
\]
For every \(\theta\in D_{n,k}\),
\[
|\operatorname{Im}\theta|
\leq\frac{1}{n+1}.
\]
Using
\[
|\sin(u+iv)|^2
=
\sin^2u+\sinh^2v
\leq\cosh^2v,
\]
we obtain
\[
|g'(\theta)|
=
2|\sin\theta|
\leq
2\cosh\!\left(\frac{1}{n+1}\right).
\]
Both \(\widehat{\theta}_k\) and \(\theta_*^{(k)}\) belong to
the convex disk \(D_{n,k}\). Integrating \(g'\) along the
segment joining these points therefore yields
\[
\left|\widehat{x}_k-x_k\right|
\leq
2\cosh\!\left(\frac{1}{n+1}\right)
\left|
\widehat{\theta}_k-\theta_*^{(k)}
\right|.
\]
Combining this inequality with
\eqref{eq:fp-aposteriori-angle} proves
\eqref{eq:fp-aposteriori-root}.
The stopping bound
\eqref{eq:fp-stopping-root-error} follows immediately.
\end{proof}
\section{Numerical Examples and Root Distributions}
\label{sec:numerical}

We illustrate the fixed-point method with the polynomials
\(\Le_5\) and \(\Le_8\). On each branch, we choose
\[
\theta_0^{(k)}=\frac{k\pi}{n+1},
\qquad k=1,\ldots,n,
\]
the center of \(D_{n,k}\).
Convergence follows from
Theorem~\ref{thm:fp-convergence}, and
Theorem~\ref{thm:fp-all-zeros} ensures that the limits
determine all the distinct zeros of \(\Le_n\). For a general treatment of fixed-point iteration and the error
analysis of iterative methods,
see~\cite{BurdenFairesBurden2016}.

The numerical computations were performed using the
\texttt{mpmath} library~\cite{mpmath} with a working
precision of 100 decimal digits.
The root-distribution plots were generated using
Matplotlib~\cite{Hunter2007}.
All computations were carried out in the Google
Colaboratory environment.

On each branch, the iteration was stopped at the first
index \(r\) for which
\begin{equation}
\label{eq:numerical-stopping}
\left|
\theta_{r+1}^{(k)}-\theta_r^{(k)}
\right|<10^{-12}.
\end{equation}
The stopping index may depend on the branch.
We denote the resulting approximations by
\[
\widehat{\theta}_k=\theta_{r+1}^{(k)},
\qquad
\widehat{x}_k=2i\cos\widehat{\theta}_k.
\]
Corollary~\ref{cor:fp-aposteriori} relates this stopping
criterion directly to an upper bound for
\(\left|\widehat{x}_k-x_k\right|\).
The bound can be evaluated without knowing the numerical
value of the exact zero \(x_k\).

The theoretical error bounds concern the truncation of
the exact fixed-point iteration and do not include
finite-precision roundoff.
The working precision of 100 decimal digits was chosen
to reduce the effects of roundoff in the computations.

Table entries are rounded to ten decimal places.
The error bounds and polynomial residuals given below
refer to the approximations before this rounding.

The residual \(\left|\Le_n(\widehat{x}_k)\right|\) is used
as an additional check of how closely the computed approximation
satisfies the polynomial equation. For background on polynomial
rootfinding and a posteriori error analysis, see
\cite{BiniFiorentino2000}. The distance to the corresponding
exact zero is controlled by Corollary~\ref{cor:fp-aposteriori}.

\begin{example}
\label{ex:numerical-le5}
\normalfont
For $n=5$, we have
\[
\Le_5(x)=2x^5+8x^3+6x-1,
\qquad
\lambda_5=-\frac{i}{2}.
\]
Thus, the iteration takes the form
\begin{equation}
\label{eq:numerical-le5-iteration}
\theta_{r+1}^{(k)}
=
\frac{
k\pi+(-1)^k
\arcsin\!\left(-\frac{i}{2}\sin\theta_r^{(k)}\right)
}{6},
\qquad k=1,\ldots,5.
\end{equation}
Theorem~\ref{thm:fp-contraction} gives the common contraction bound
$\kappa_5=\sqrt{61}/48\approx0.1627135349$.

By Proposition~\ref{prop:fp-symmetry}, the middle branch
\(k=3\) yields the unique positive real zero.
Starting from $\theta_0^{(3)}=\pi/2$, write
\[
\theta_r^{(3)}=\frac{\pi}{2}+iy_r,
\qquad y_0=0.
\]
Using $\sin(\pi/2+iy)=\cosh y$ and
$\arcsin(-it)=-i\operatorname{arsinh}(t)$ for real $t$,
equation~\eqref{eq:numerical-le5-iteration} reduces to the real
iteration
\[
y_{r+1}
=
\frac16\operatorname{arsinh}\!\left(\frac12\cosh y_r\right),
\qquad y_0=0,
\]
where $\operatorname{arsinh}(t)=\log(t+\sqrt{1+t^2})$.
The first four iterates are
\[
\begin{aligned}
y_1&\approx0.0802019708433,
&\qquad y_2&\approx0.0804417419131,\\
y_3&\approx0.0804431779994,
& y_4&\approx0.0804431866136.
\end{aligned}
\]
The limit satisfies $y_*\approx0.0804431866656$, and hence
\[
\theta_*^{(3)}=\frac{\pi}{2}+iy_*,
\qquad
x_3=2\sinh y_*\approx0.1610599483.
\]
In particular, $0<x_3<1/6$, in agreement with
Theorem~\ref{thm:positive-root-bounds}.

At the fourth iterate, the a priori estimate
\eqref{eq:fp-error} gives
\[
\left|
\theta_4^{(3)}-\theta_*^{(3)}
\right|
\leq
\frac{\kappa_5^4}{1-\kappa_5}
\left|
\theta_1^{(3)}-\theta_0^{(3)}
\right|.
\]
The right-hand side is approximately
\(6.71438\times10^{-5}\).

For the final approximations obtained when
\eqref{eq:numerical-stopping} is satisfied,
Corollary~\ref{cor:fp-aposteriori} gives
\[
\max_{1\leq k\leq5}
\left|\widehat{x}_k-x_k\right|
<
2\cosh\!\left(\frac16\right)
\frac{\kappa_5}{1-\kappa_5}\,10^{-12}
<
3.95\times10^{-13}.
\]
This is a truncation error bound for the final zero
approximations, evaluated before rounding the table entries.

The results for all five branches are listed in
Table~\ref{tab:numerical-le5}. The zeros consist of one positive
real zero and two nonreal conjugate pairs, with
$x_5=\overline{x_1}$ and $x_4=\overline{x_2}$.
Their locations are shown in Figure~\ref{fig:numerical-le5}.
The maximum residual obtained with the stated stopping criterion is
\[
\max_{1\leq k\leq5}
\left|\Le_5(\widehat{x}_k)\right|
\approx2.08818\times10^{-13}.
\]

\begin{center}
\captionof{table}{Initial values, approximate fixed points, and approximate
zeros of $\Le_5$.}
\label{tab:numerical-le5}

\small
\setlength{\tabcolsep}{4pt}
\renewcommand{\arraystretch}{1.15}
\begin{tabular}{c c c c}
\hline
$k$ & $\theta_0^{(k)}$ & $\widehat{\theta}_k$ & $\widehat{x}_k$ \\
\hline
1 & $0.5235987756$
  & $0.5207160272+0.0410823591i$
  & $0.0408885611+1.7363906309i$ \\
2 & $1.0471975512$
  & $1.0445007726-0.0701530395i$
  & $-0.1214185353+1.0071405391i$ \\
3 & $1.5707963268$
  & $1.5707963268+0.0804431867i$
  & $0.1610599483$ \\
4 & $2.0943951024$
  & $2.0970918810-0.0701530395i$
  & $-0.1214185353-1.0071405391i$ \\
5 & $2.6179938780$
  & $2.6208766264+0.0410823591i$
  & $0.0408885611-1.7363906309i$ \\
\hline
\end{tabular}
\end{center}

\begin{center}
\includegraphics[width=0.50\textwidth]{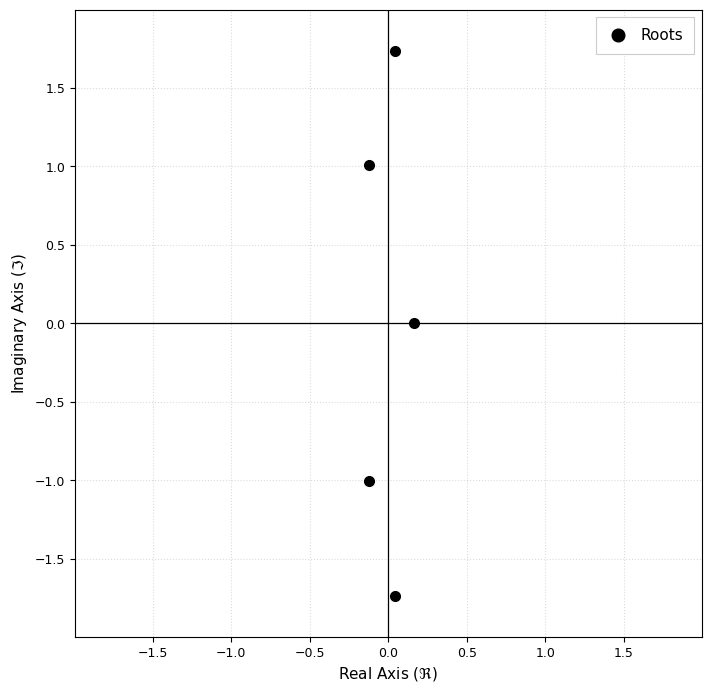}
\captionof{figure}{Zeros of $\Le_5$: one positive real zero and two nonreal
conjugate pairs.}
\label{fig:numerical-le5}
\end{center}

\end{example}

\begin{example}
\label{ex:numerical-le8}
\normalfont

For $n=8$, the polynomial is
\[
\Le_8(x)=2x^8+14x^6+30x^4+20x^2+1.
\]
Since $\lambda_8=1/2$, the iteration becomes
\[
\theta_{r+1}^{(k)}
=
\frac{
k\pi+(-1)^k
\arcsin\!\left(\frac12\sin\theta_r^{(k)}\right)
}{9},
\qquad k=1,\ldots,8.
\]
The common contraction bound is
$\kappa_8=\sqrt{61}/72\approx0.1084756899$.
Starting with $\theta_0^{(1)}=\pi/9$, the first branch gives
\[
\begin{aligned}
\theta_1^{(1)}&\approx0.3299708780,
&\qquad\theta_2^{(1)}&\approx0.3309853024,\\
\theta_3^{(1)}&\approx0.3309312793,
&\theta_4^{(1)}&\approx0.3309341560.
\end{aligned}
\]
The limiting fixed point and its corresponding zero are
\[
\theta_*^{(1)}\approx0.3309340105,
\qquad
x_1=2i\cos\theta_*^{(1)}\approx1.8914785426i.
\]
At the fourth iterate, the a priori estimate
\eqref{eq:fp-error} gives
\[
\left|
\theta_4^{(1)}-\theta_*^{(1)}
\right|
\leq
\frac{\kappa_8^4}{1-\kappa_8}
\left|
\theta_1^{(1)}-\theta_0^{(1)}
\right|.
\]
The right-hand side is approximately
\(2.96562\times10^{-6}\).

For the final approximations obtained when
\eqref{eq:numerical-stopping} is satisfied,
Corollary~\ref{cor:fp-aposteriori} gives
\[
\max_{1\leq k\leq8}
\left|\widehat{x}_k-x_k\right|
<
2\cosh\!\left(\frac19\right)
\frac{\kappa_8}{1-\kappa_8}\,10^{-12}
<
2.45\times10^{-13}.
\]
This is a truncation error bound for the final zero
approximations, evaluated before rounding the table entries.

Table~\ref{tab:numerical-le8} lists the approximations for all eight branches.
The iterates remain real because the initial values are real and
the arguments of $\arcsin$ lie in \([-1/2,1/2]\).
Proposition~\ref{prop:fp-symmetry} gives \(x_{9-k}=-x_k\), \(k=1,\ldots,4\).
Thus, the zeros form four symmetric pairs on the imaginary axis,
with imaginary parts in \((-2,2)\), in agreement with
Theorem~\ref{thm:even-all-imaginary} and as illustrated in Figure~\ref{fig:numerical-le8}.
The maximum polynomial residual, evaluated before rounding, is
\[
\max_{1\le k\le 8}
\left|\mathrm{Le}_8(\widehat{x}_k)\right|
\approx 5.42560\times10^{-13}.
\]

\begin{center}
\captionof{table}{Initial values, approximate fixed points, and approximate
zeros of $\Le_8$.}
\label{tab:numerical-le8}

\small
\renewcommand{\arraystretch}{1.15}
\begin{tabular}{c c c c}
\hline
$k$ & $\theta_0^{(k)}$ & $\widehat{\theta}_k$ & $\widehat{x}_k$ \\
\hline
1 & $0.3490658504$ & $0.3309340105$ & $1.8914785426i$ \\
2 & $0.6981317008$ & $0.7361737838$ & $1.4820862364i$ \\
3 & $1.0471975512$ & $0.9989806471$ & $1.0823195618i$ \\
4 & $1.3962634016$ & $1.4540045207$ & $0.2330529481i$ \\
5 & $1.7453292520$ & $1.6875881329$ & $-0.2330529481i$ \\
6 & $2.0943951024$ & $2.1426120065$ & $-1.0823195618i$ \\
7 & $2.4434609528$ & $2.4054188698$ & $-1.4820862364i$ \\
8 & $2.7925268032$ & $2.8106586431$ & $-1.8914785426i$ \\
\hline
\end{tabular}
\end{center}

The numerical values in Table~\ref{tab:numerical-le8} show that the
eight zeros are purely imaginary and occur in four symmetric pairs.
Their distribution in the complex plane is shown in
Figure~\ref{fig:numerical-le8}.

\begin{center}
\includegraphics[width=0.45\textwidth]{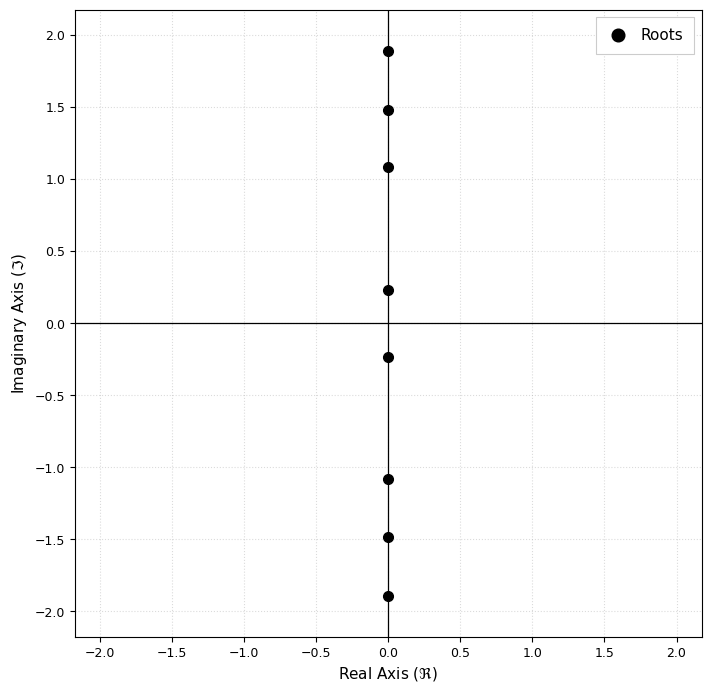}
\captionof{figure}{The eight zeros of $\Le_8$ form four symmetric pairs
on the imaginary axis.}
\label{fig:numerical-le8}
\end{center}

This numerical configuration agrees with the theoretical description
of the zeros obtained above.

\end{example}

Figure~\ref{fig:numerical-higher-degrees} displays the zeros of
$\Le_{12}$, $\Le_{47}$, and $\Le_{99}$ obtained by the same fixed-point
construction. All zeros of $\Le_{12}$ lie on the imaginary axis,
whereas each of the odd-degree polynomials has one positive real
zero and nonreal conjugate pairs. The small real parts in the
odd-degree cases are consistent with the bound
\[
|\operatorname{Re}z|
\leq2\sinh\!\left(\frac{1}{n+1}\right)
\]
from Theorem~\ref{thm:fp-strip}.
The horizontal scales in panels~\subref{fig:numerical-le47}
and~\subref{fig:numerical-le99} are enlarged to make these real
parts visible. Therefore, the apparent widths of the distributions
should be interpreted using the numerical axis scales.

\begin{figure}[!htbp]
\centering

\begin{subfigure}[t]{0.32\textwidth}
\centering
\includegraphics[
  width=\linewidth,
  height=0.96\linewidth,
  keepaspectratio
]{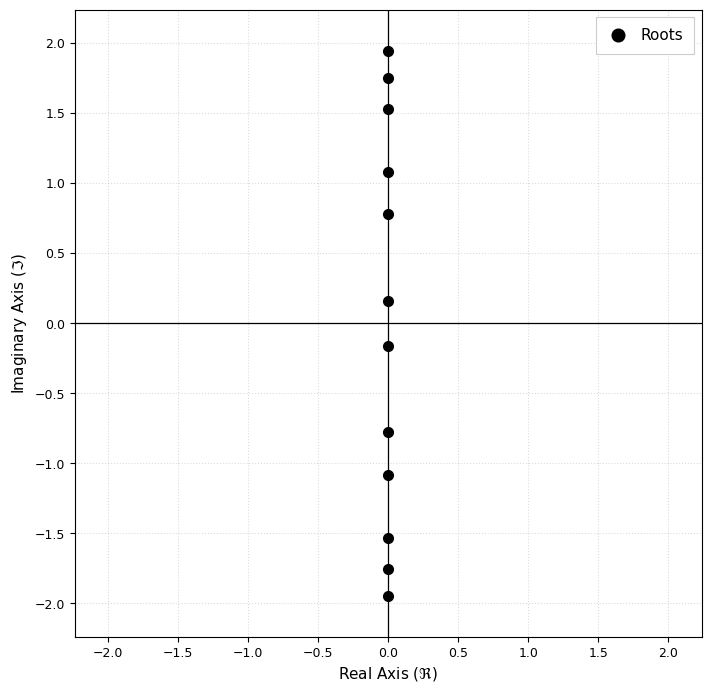}
\caption{$\Le_{12}$}
\label{fig:numerical-le12}
\end{subfigure}\hfill
\begin{subfigure}[t]{0.32\textwidth}
\centering
\includegraphics[
  width=\linewidth,
  height=0.96\linewidth,
  keepaspectratio
]{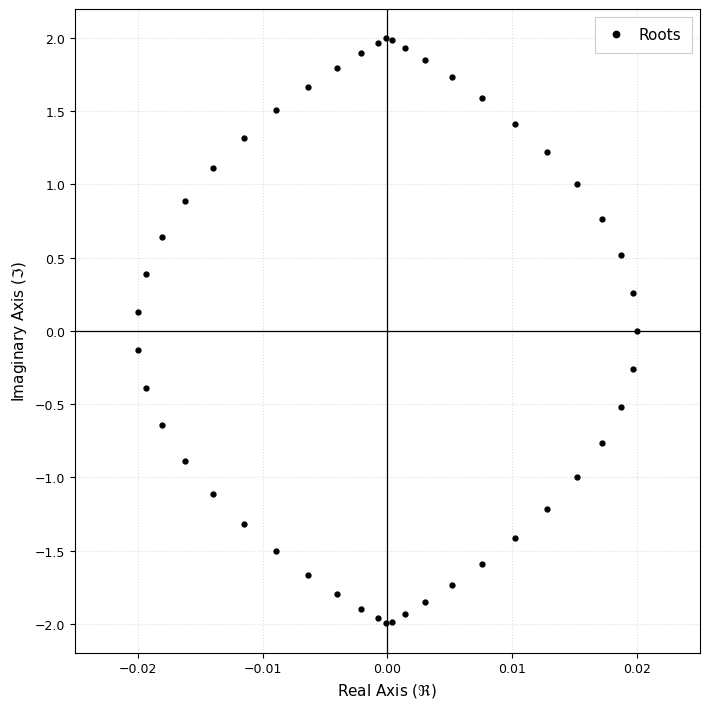}
\caption{$\Le_{47}$}
\label{fig:numerical-le47}
\end{subfigure}\hfill
\begin{subfigure}[t]{0.32\textwidth}
\centering
\includegraphics[
  width=\linewidth,
  height=0.96\linewidth,
  keepaspectratio
]{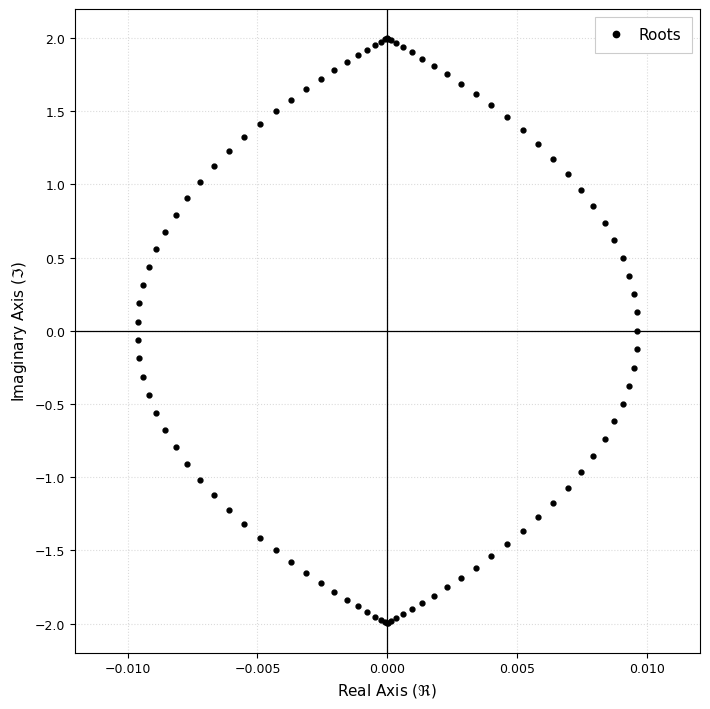}
\caption{$\Le_{99}$}
\label{fig:numerical-le99}
\end{subfigure}

\caption{Root distributions for three higher degrees.
Black points represent the zeros.
In panels (b) and (c), the real and imaginary axes use different
scales.}

\label{fig:numerical-higher-degrees}
\end{figure}

\section{Conclusion}
The fixed-point formulation developed in this work connects
the location and separation of Leonardo polynomial zeros
with their numerical approximation. By constructing a distinct
invariant disk for each corresponding angular parameter,
the method provides a constructive description of the complete
zero set together with convergence guarantees and explicit
error estimates.

A natural direction for further research is to extend this
approach to generalized Leonardo polynomials, including
bivariate families. In particular, it would be useful to
identify parameter ranges for which suitable invariant domains
and contraction estimates can be established, and to investigate
how the parameters affect the location and multiplicity of
the zeros.

\Authorcontributions{}{All the co-authors have contributed equally in all aspects of the preparation of this submission.}

\Competinginterests{}{The authors declare that they have no known competing financial interests or personal relationships that could have appeared to
	influence the work reported in this paper.}


\begin{thebibliography}{99}
\bibitem{AlpKocer2021} Alp, Y. and Koçer, E. G. (2021). Some properties of Leonardo numbers. Konuralp Journal of Mathematics 9(1):183--189

\bibitem{AlpKocer2025} Alp, Y. and Koçer, E. G. (2025). Bivariate Leonardo polynomials and Riordan arrays. Notes on Number Theory Discrete Mathematics 31(2):236--250. DOI: 10.7546/nntdm.2025.31.2.236-250

\bibitem{Bartle and Sherbert2011} Bartle, R. G. and Sherbert, D. R. (2011). Introduction to Real Analysis. John Wiley \& Sons, 4th edition.

\bibitem{Berinde2007} Berinde, V. (2007). Iterative Approximation of Fixed Points, volume 1912 of Lecture Notes in Mathematics. Springer, Berlin, Heidelberg, second edition. DOI: 10.1007/978-3-540-72234-2

\bibitem{BiniFiorentino2000} Bini, D. A. and Fiorentino, G. (2000). Design, analysis, and implementation of a multiprecision polynomial rootfinder. Numerical Algorithms 23:127–173. DOI: 10.1023/A:1019199917103

\bibitem{Brown and Churchill2014} Brown, J. W. and Churchill, R. V. (2014). Complex Variables and Applications. McGraw-Hill Education, 9th edition.

\bibitem{BurdenFairesBurden2016} Burden, R. L., Faires, J. D., and Burden, A. M. (2016). Numerical Analysis. Cengage
Learning, 10th edition.

\bibitem{Catarino and Borges2020} Catarino, P. and Borges, A. (2020). On Leonardo numbers. Acta Mathematica
Universitatis Comenianae 89(1):75–86.

\bibitem{FrontczakGoy2021} Frontczak, R. and Goy, T. (2021). Chebyshev–Fibonacci polynomial relations
using generating functions. Integers 21:A100. DOI: 10.5281/zenodo.10839350

\bibitem{HoggattBicknell1973} Hoggatt, Jr., V. E. and Bicknell, M. (1973). Roots of Fibonacci polynomials. The Fibonacci Quarterly 11(3):271--274. DOI: 10.1080/00150517.1973.12430825

\bibitem{Hunter2007} Hunter, J. D. (2007). Matplotlib: A 2D graphics environment. Computing in Science \& Engineering 9(3):90--95. DOI: 10.1109/MCSE.2007.55

\bibitem{Kaymak2018} \"{O}ztunç Kaymak, \"{O}. (2018). Some remarks on the zeros of tribonacci polynomials. International Journal of Analysis and Applications 16(3):368–373. DOI:10.28924/2291-8639-16-2018-368

\bibitem{Koshy2001} Koshy, T. (2001). Fibonacci and Lucas Numbers with Applications. John Wiley \& Sons, Inc., 1st edition. DOI: 10.1002/9781118033067

\bibitem{KumariPrasadSingh2026} Kumari, M., Prasad, K., and Singh, M. K. (2026). On generalized metallic Leonardo numbers: Silver, Bronze, and Copper cases. Universal Journal of Mathematics and Applications 9(1):44--52. DOI: 10.32323/ujma.1854795

\bibitem{MasonHandscomb2002} Mason, J. C. and Handscomb, D. C. (2002). Chebyshev Polynomials. Chapman and Hall/CRC, 1st edition. DOI: 10.1201/9781420036114

\bibitem{KaymakOzgur2022} \"{O}ztunç Kaymak, \"{O}. and \"{O}zg\"{u}r, N. (2022). On the zeros of \(R\)-bonacci polynomials and their derivatives. Communications Faculty of Sciences University of Ankara Series A1 Mathematics and Statistics 71(4):978--992. DOI: 10.31801/cfsuasmas.1037229

\bibitem{PrasadKumari2025} Prasad, K. and Kumari, M. (2025). The Leonardo polynomials and their algebraic
properties. Proceedings of the Indian National Science Academy 91:662–671. DOI: 10.1007/s43538-024-00348-0

\bibitem{PrasadKumariSingh2026} Prasad, K., Kumari, M., and Singh, M. K. (2026). A note on Gaussian family of the
Leonardo polynomials. Integral Transforms and Special Functions, pages 1–17. DOI: 10.1080/10652469.2026.2629501

\bibitem{Rivlin2020} Rivlin, T. J. (2020). Chebyshev Polynomials: From Approximation Theory to Algebra and
Number Theory. Dover Publications, 2nd edition.

\bibitem{SmajlovicSabanacSceta2025} Smajlovi\'{c}, L., \v{S}abanac, Z., and \v{S}\'{c}eta, L. (2025). Relations between Chebyshev, Fibonacci and Lucas polynomials via trigonometric sums. The Fibonacci Quarterly, 63(2):439--455. DOI: 10.1080/00150517.2025.2460552

\bibitem{Soykan2021} Soykan, Y. (2021). Generalized Leonardo numbers. Journal of Progressive Research in
Mathematics, 18(4):58–84.

\bibitem{Soykan2026} Soykan, Y. (2026). Leonardo polynomials and numbers: Solutions, linearizations, and generating functions. International Journal of Advances in Applied Mathematics and Mechanics, 13(4):80--222. DOI: 10.26541/ijaamm.2026.130408

\bibitem{mpmath} The mpmath development team (2026). mpmath: a Python library
for arbitrary-precision floating-point arithmetic.

\bibitem{TugluEkinciKeskin2026} Tu\u{g}lu, N., Ekinci, S. S., and \c{C}etin Keskin, M. (2026). \(q\)-Leonardo polynomials. Notes on Number Theory and Discrete Mathematics, 32(2):281--299. DOI: 10.7546/nntdm.2026.32.2.281-299

\bibitem{Vajda1989} Vajda, S. (1989). Fibonacci \& Lucas Numbers, and the Golden Section: Theory and
Applications. Ellis Horwood Ltd. and Halsted Press.

\bibitem{Wang2004} Wang, X. (2004). A simple proof of Descartes's rule of signs. The American Mathematical Monthly, 111(6):525--526. DOI: 10.1080/00029890.2004.11920108

\bibitem{Yagmur2023} Ya\u{g}mur, T. (2023). On Gaussian Leonardo hybrid polynomials. Symmetry, 15(7):1422. DOI: 10.3390/sym15071422

\end{thebibliography}
\end{document}